\documentclass[paper=a4,english,fontsize=11pt,parskip=half,abstract=true]{scrartcl}
\usepackage{babel}
\usepackage[utf8]{inputenc}
\usepackage[T1]{fontenc}
\usepackage[left=20mm,right=20mm,top=30mm,bottom=30mm]{geometry}
\usepackage{amsmath}
\usepackage{amsthm}
\usepackage{amssymb}
\usepackage{enumerate} 
\usepackage{thmtools}
\usepackage{mathtools}
\mathtoolsset{centercolon} 
\usepackage[bookmarks=true,
            pdftitle={Characterizing inner automorphisms and realizing outer automorphisms},
            pdfauthor={Benjamin Sambale},
            pdfkeywords={},
            pdfstartview={FitH}]{hyperref}

\newtheorem{Thm}{Theorem} 
\newtheorem*{Thm*}{Theorem}
\newtheorem{Lem}[Thm]{Lemma}

\numberwithin{equation}{section}

\allowdisplaybreaks[1]

\renewcommand{\phi}{\varphi}
\newcommand{\C}{\mathrm{C}}

\newcommand{\Z}{\mathrm{Z}}

\newcommand{\ZZ}{\mathbb{Z}}

\newcommand{\QQ}{\mathbb{Q}}

\newcommand{\NN}{\mathbb{N}}
\newcommand{\FF}{\mathbb{F}}
\newcommand{\Aut}{\operatorname{Aut}}
\newcommand{\Inn}{\operatorname{Inn}}
\newcommand{\Out}{\operatorname{Out}}
\newcommand{\GL}{\operatorname{GL}}

\newcommand{\id}{\operatorname{id}}

\title{Characterizing inner automorphisms\\ and realizing outer automorphisms} 
\author{Benjamin Sambale\footnote{Institut für Algebra, Zahlentheorie und Diskrete Mathematik, Leibniz Universität Hannover, Welfengarten 1, 30167 Hannover, Germany,
\href{mailto:sambale@math.uni-hannover.de}{sambale@math.uni-hannover.de}}}
\date{\today}

\begin{document}
\frenchspacing
\maketitle
\renewcommand{\sectionautorefname}{Section}

\begin{abstract}\noindent
We give elementary proofs of the following two theorems on automorphisms of a finite group $G$: (1) An automorphism of $G$ is inner if and only if it extends to an automorphism of every finite group containing $G$.
(2) There exists a finite group, whose outer automorphism group is isomorphic to $G$. 
The first theorem was proved by Pettet using a graph-theoretical construction of Heineken--Liebeck. A Lie-theoretical proof of the second theorem was sketched by Cornulier in a MathOverflow post. Our proofs are purely group-theoretical.
\end{abstract}


\section{Introduction}

An automorphism $\alpha$ of a group $G$ is called \emph{inner} if there exists some $g\in G$ such that $\alpha(x)=gxg^{-1}$ for all $x\in G$. If $G$ is a subgroup of a group $H$, it is clear that $g$ still induces an (inner) automorphism of $H$. In 1987, Schupp~\cite{Schupp} has shown conversely that inner automorphisms are characterized by this property, i.\,e. if $\alpha\in\Aut(G)$ extends to every group containing $G$, then $\alpha$ is inner. According to \cite{DugasGobel}, this has answered a question of Macintyre. The question was asked again much later by Bergman~\cite[p. 93]{Bergman}, who obtained a partial answer in the language of category theory.
Using free products and small cancellation theory, Schupp constructs for a non-inner $\alpha\in\Aut(G)$ an infinite group $H$ such that $\alpha$ does not extend to $H$ (if $G$ is countable, the construction is already contained in Miller--Schupp~\cite{MillerSchupp}). 
One may asks whether inner automorphisms of \emph{finite} groups $G$ are characterized by the property that they extend to all \emph{finite} groups containing $G$.

The first step in this direction was a paper from 1974 of Heineken--Liebeck~\cite{HeinekenLiebeck}, who constructed a finite $p$-group $P$ such that the image of the canonical map $\Aut(P)\to\Aut(P/\Z(P))$ is isomorphic to $G$. Their construction relies on a variation of Frucht's theorem on the automorphism group of graphs, and requires a treatment of special cases. 
Using more advanced graph theoretical theorems, Lawton~\cite{Lawton} came up with a shorter proof. Subsequently, Webb~\cite{WebbAut} has refined the construction (again at the cost of more graph theory) to obtain a special $p$-group $P$ (that means $P'=\Z(P)=\Phi(P)$ is elementary abelian) with the desired property (see also Hughes~\cite{Hughes}). 
Only after Schupp's paper, Pettet~\cite{Pettet} noticed in 1989 that this result implies Schupp's theorem for finite groups (and more restricted families of groups). In the same paper, he also obtained the dual statement for factor groups (see the theorem below).
Eventually, Pettet~\cite{Pettet2} gave a new proof of Schupp's original theorem with the same graph theoretical approach. An alternative construction of $P$ was established by Bryant--Kov\'{a}cs~\cite{Bryant} in 1978 making use of Lie theory (see also Huppert--Blackburn~\cite[Theorem~VIII.13.5]{Huppert2} and Hartley--Robinson~\cite{HartleyRobinson}). 

The first objective of this paper is to provide a new elementary proof of the following theorem, which avoids small cancellation theory, graph theory and Lie theory. 

\begin{Thm*}[\textsc{Pettet}]
For every automorphism $\alpha$ of a finite group $G$ the following statements are equivalent:
\begin{enumerate}[(1)]
\item $\alpha$ is an inner automorphism.
\item $\alpha$ extends to every finite group containing $G$.
\item $\alpha$ lifts to every finite group $\hat{G}$ such that $\hat{G}/N\cong G$ for some characteristic subgroup $N\le \hat{G}$.
\end{enumerate}
\end{Thm*}

The strategy is to replace the $p$-group $P$ in \cite{HeinekenLiebeck} by a semidirect product $N=Q\rtimes P$, where $Q$ is an elementary abelian $q$-group and $P$ is a $p$-group of nilpotency class $2$. In contrast to the papers cited above, we do not make an effort to minimize $N$. 

Our second objective concerns the inverse problem on automorphism groups. By a theorem of Leder\-mann--Neumann~\cite{LedermannNeumann}, there exist only finitely many finite groups with a given automorphism group (for an elementary proof see \cite{SambaleAut}).
However, not every group actually occurs as an automorphism group. For instance, it is a popular (and easy) exercise that a non-trivial cyclic group of odd order cannot be an automorphism group. The situation changes if we instead consider the \emph{outer} automorphism group $\Out(H):=\Aut(H)/\Inn(H)$ of groups $H$. Indeed, Matumoto~\cite{Matumoto} proved that for every group $G$ there exists a group $H$ such that $\Out(H)\cong G$ (for finite $G$ this was established earlier by Kojima~\cite{Kojima}). Similar to Schupp's paper, the construction of $H$ is based on an HNN-extension and yields infinite groups. 
In later work, it was shown that $H$ can be chosen to be locally finite, finitely generated (if $G$ is countable), metabelian or simple (see \cite{BraunGobel,BumaginWise,GobelParas,DGG}). Problem~16.59 in the Kourovka Notebook~\cite{Kourovka} has asked if $H$ can be chosen finite when $G$ is finite. This was answered in 2020 by Cornulier on MathOverflow~\cite{Cornulier}.

\begin{Thm*}[\textsc{Cornulier}]
Every finite group is the outer automorphism group of some finite group.
\end{Thm*}

His proof uses Lie algebras and is not easy to follow as it is written backwards. In \autoref{secOut} we present a purely group-theoretical proof based on Cornulier's ideas. As in Pettet's theorem, the idea is to construct a semidirect product $P\rtimes Q$, but this time $Q$ is abelian and $P$ is a $p$-group of exponent $p$ with a certain nilpotency class. Our group is slightly smaller compared to Cornulier's construction.

\section{Characterizing inner automorphisms}

Our first lemma is a prototype of \autoref{thmext}. 

\begin{Lem}\label{lemaut}
Let $G$ be a group acting faithfully on a cyclic group $N$. Then every automorphism of $\hat{G}:=N\rtimes G$ normalizing $N$, centralizes $\hat{G}/N$. 
\end{Lem}
\begin{proof}
We identify $N=\langle x\rangle$ and $G$ with the natural subgroups of $\hat{G}$. Let $\alpha\in\Aut(\hat{G})$ normalizing $N$. 
For a fixed $g\in G$ there exist $s,t\in\ZZ$ with $gxg^{-1}=x^s$ and $\alpha(x)=x^t$. Hence,
\[gx^tg^{-1}=x^{st}=\alpha(x)^s=\alpha(x^s)=\alpha(gxg^{-1})=\alpha(g)x^t\alpha(g)^{-1}.\]
Since $G$ acts faithfully on $N=\langle x^t\rangle$, we obtain $\alpha(g)\equiv g\pmod{N}$ as desired.
\end{proof}

In the following we develop some elementary facts of finitely presented groups. 
For elements $x_1,x_2,\ldots $ of a group $G$ we define commutators by $[x_1,x_2]=x_1x_2x_1^{-1}x_2^{-1}$ and $[x_1,\ldots,x_k]:=[x_1,[x_2,\ldots,x_k]]$ for $k\ge 3$. The commutator subgroup of $G$ is denoted by $G'$.  

\begin{Lem}\label{lemp3}
Let $p$ be a prime and $a,b\in\ZZ$. Then 
\[P:=\langle x,y\mid [x,x,y]=[y,x,y]=1,\ x^p=[x,y]^a,\ y^p=[x,y]^b\rangle\]
is a non-abelian group of order $p^3$.
\end{Lem}
\begin{proof}
The commutator relations show that $P$ has nilpotency class $\le 2$, i.\,e. $P'\le\Z(P)$. Hence, $[x,y]^p=[x^p,y]=1$ and $|P'|\le p$ (see \cite[Hilfssatz~III.1.3]{Huppert}). It follows that $|P|\le p^3$. To complete the proof, we construct a group of order $p^3$ realizing the given relations. Suppose first that $p=2$. If $ab$ is even, then $P\cong D_8$ and otherwise $P\cong Q_8$. Thus, let $p>2$. If $a\equiv b\equiv 0\pmod{p}$, then the extraspecial group of exponent $p$ fulfills the relations as is well-known. Now let $p\nmid a$ and $a'\in\ZZ$ such that $aa'\equiv -b\pmod{p}$. For $y':=x^{a'}y$ we compute
\[(y')^p=x^{pa'}y^p[y,x^{a'}]^{\binom{p}{2}}=x^{pa'}y^p=[x,y]^{aa'+b}=1\] 
and $[x,y']=[x,y]$ by \cite[Hilfssatz~III.1.3]{Huppert}.
Hence, replacing $y$ by $y'$ leads to $b=0$. This remains true when we replace $y$ by $y^{-a}$. Then $x^p=[x,y]^{-1}=[y,x]$ and 
\[P\cong\langle x,y\mid x^{p^2}=y^p=1,\ yxy^{-1}=x^{1+p}\rangle\cong C_{p^2}\rtimes C_p.\qedhere\]
\end{proof}

\begin{Lem}\label{lemgenrel}
Let $F$ be the free group in the free generators $x_1,\ldots,x_n$. Let $c_1,\ldots,c_n\in F'$. For a prime $p$, let $P$ be the group generated by $x_1,\ldots,x_n$ subject to the relations
$x_i^p=c_i$ and $[x_i,x_j,x_k]=1$ for all $1\le i,j,k\le n$. Then $P/P'$ is an elementary abelian $p$-group with basis $\{x_iP':i=1,\ldots,n\}$ and $P'$ is an elementary abelian $p$-group with basis $\{[x_i,x_j]:1\le i<j\le n\}$. In particular, $|P|=p^{\binom{n+1}{2}}$.
\end{Lem}
\begin{proof}
It follows from $x_i^p=w_i\in P'$ that $P/P'$ is an elementary abelian $p$-group generated by $\{x_iP':i=1,\ldots,n\}$.
Since $[x_i,x_j,x_k]=1$ for all $i,j,k$, $P$ has nilpotency class $\le 2$. Hence, $[xy,z]=[x,z][y,z]$ and $[x,yz]=[x,y][x,z]$ for all $x,y,z\in P$ (see \cite[Hilfssatz~III.1.2]{Huppert}). In particular, $[x_i,x_j]^p=[x_i^p,x_j]=[c_i,x_j]=1$. This shows that $P'$ is an elementary abelian $p$-group generated by $\{[x_i,x_j]:1\le i<j\le n\}$. 

Suppose that $x:=\prod_{i<j}[x_i,x_j]^{a_{ij}}=1$ for some integers $0\le a_{ij}\le p-1$. We fix $i<j$ and consider the free group $F_2$ generated by $y_i$ and $y_j$. Let $\phi:F\to F_2$ be the homomorphism defined by
\[\phi(x_k)= \begin{cases}
y_k&\text{if }k\in \{i,j\},\\
1&\text{otherwise}
\end{cases}\]
for $k=1,\ldots,n$. Set 
\[P_2:=\langle y_i,y_j\mid [y_i,y_i,y_j]=[y_j,y_i,y_j]=1,\ y_i^p=\phi(c_i),\ y_j^p=\phi(c_j)\rangle.\]
As elements of $P_2$, $\phi(c_i)$ and $\phi(c_j)$ are (possibly trivial) powers of $[y_i,y_j]$. 
Thus by \autoref{lemp3}, $P_2$ is a non-abelian group of order $p^3$. Since every relation of $P$ in the $x_k$ is satisfied by a relation of $P_2$ in the $y_k$, $\phi$ factors through a homomorphism $\overline{\phi}:P\to P_2$. It follows that $[y_i,y_j]^{a_{ij}}=\overline{\phi}(x)=1$ and $a_{ij}=0$. This shows that the commutators $[x_i,x_j]$ with $i<j$ are linearly independent, so they form a basis of $P'$. 
Now let $x:=\prod_{i=1}^nx_i^{a_i}\in P'$ for some $0\le a_i\le p-1$. Then 
\[1=[x_1,x]=\prod_{i=2}^n[x_1,x_i]^{a_i}\]
and $a_i=0$ for $i=2,\ldots,n$. Similarly, we obtain $a_1=0$. Therefore, $P/P'$ has rank $n$ as claimed. 
Finally, $|P|=|P'||P/P'|=p^{\binom{n}{2}+n}=p^{\binom{n+1}{2}}$.
\end{proof}

\begin{Thm}\label{thmext}
For every finite group $G$ there exist primes $q>p>|G|$ and a finite $\{p,q\}$-group $N$ such that every automorphism of $\hat{G}:=N\rtimes G$ induces an inner automorphism of $\hat{G}/N\cong G$. 
\end{Thm}
\begin{proof}
Without loss of generality, we assume that $G\ne 1$.
Let $p>|G|$ be a prime. Let $x_1,\ldots,x_n\in G$ be a generating set of $G$ not containing $1$. Let $P$ be the $p$-group with generators $\{v_g:g\in G\}$ and relations
\[
[v_g,v_h,v_k]=1,\qquad v_g^p=\prod_{i=1}^n[v_g,v_{gx_i}]^i
\]
for all $g,h,k\in G$. By \autoref{lemgenrel}, $P/P'$ is elementary abelian of rank $|G|$ and $P'$ is elementary abelian with basis $[v_g,v_h]$ where $g<h$ for some fixed total order on $G$. For a fixed $g\in G$, the elements $\{v_{gh}:h\in G\}$ fulfill the same relations as the $v_h$. Thus, there exists an automorphism $\phi_g\in\Aut(P)$ with $\phi_g(v_h)=v_{gh}$ for all $h\in G$. This gives rise to a regular action $\phi:G\to\Aut(P)$, $g\mapsto\phi_g$. 

By an elementary special case of Dirichlet's theorem (see \cite[Theorem~3.1.12]{FineRosenberger}), there exists a prime $q$ such that $p\mid q-1$. Let $Q$ be the elementary abelian $q$-group with basis $\{w_g:g\in G\}$. Let $w_g\mapsto w_g^\zeta$ be an automorphism of order $p$ of $\langle w_g\rangle$. For $g\in G$, define $\gamma_g\in\Aut(Q)$ by
\[\gamma_g(w_h):=\begin{cases}
w_g^\zeta&\text{if }h=g,\\
w_h&\text{if }h\ne g.
\end{cases}\qquad(h\in G)\]
Let $\gamma:P\to\Aut(Q)$, $v_g\mapsto\gamma_g$ be the homomorphism with kernel $P'$. This gives rise to the semidirect product $N:=Q\rtimes P$ with $\Z(N)=P'$. As usual, we identify $P$ and $Q$ with the natural subgroups of $N$. Then $v_gw_hv_g^{-1}=\gamma_g(w_h)$ for all $g,h\in G$.
Again, we have a regular action $\psi:G\to\Aut(Q)$, $g\mapsto\psi_g$ with $\psi_g(w_h)=w_{gh}$. Moreover, $\phi$ and $\psi$ are compatible in the sense that
\[\phi_g(v_h)\psi_g(w_k)\phi_g(v_h)^{-1}=v_{gh}w_{gk}v_{gh}^{-1}=\gamma_{gh}(w_{gk})=\psi_g(\gamma_h(w_k))=\psi_g(v_hw_kv_h^{-1})\]
for all $g,h,k\in G$. In this way, $G$ acts faithfully on $N$. As before, we identify $G$ and $N$ with subgroups of $\hat{G}:=N\rtimes G$. Then $gv_hg^{-1}=v_{gh}$ and $gw_hg^{-1}=w_{gh}$ for $g,h\in G$. 

Now let $\alpha\in\Aut(\hat{G})$. Since $q>p>|G|$, $\alpha$ normalizes the normal Sylow $q$-subgroup $Q$, the normal Hall subgroup $N$, and in turn $\Z(N)=P'$. By the Schur--Zassenhaus theorem, there exists some $y\in N$ such that $\alpha(G)=yGy^{-1}$ (we do not require the Feit--Thompson theorem, because $N$ is solvable). Since $y$ centralizes $\hat{G}/N$, we may compose $\alpha$ with the inner automorphism induced by $y^{-1}$. Then $\alpha$ normalizes $G$.  
Next, we consider the action of $\alpha$ on 
\[N/P'\cong\bigtimes_{g\in G} \langle w_g, v_g\rangle\cong (C_q\rtimes C_p)^{|G|}.\]
It follows from
\[|\C_{N/P'}(\alpha(v_1)P')|=|\C_{N/P'}(v_1P')|=p^{|G|}q^{|G|-1}\]
that $\alpha(v_1)\in\langle v_g,w_g\rangle P'$ for some $g\in G$. Composing $\alpha$ with the inner automorphism induced by $g^{-1}$, we may assume that $g=1$. Then $\alpha$ induces an automorphism of $\langle v_1,w_1\rangle P'/P'\cong C_q\rtimes C_p$. 
By \autoref{lemaut}, there exists $t_1\in P'$ such that $\alpha(v_1)\equiv v_1t_1\pmod{Q}$. Moreover, for every $g\in G$ there exists $t_g\in P'$ with
\[\alpha(v_g)=\alpha(g v_1 g^{-1})\equiv \alpha(g)v_1t_1\alpha(g)^{-1}\equiv v_{\alpha(g)}t_g\pmod{Q}.\]
Consequently,
\[\prod_{i=1}^n[v_1,v_{x_i}]^i=v_1^p=(v_1t_1)^p\equiv \alpha(v_1)^p\equiv \prod_{i=1}^n[v_1t_1,v_{\alpha(x_i)}t_{x_i}]^i\equiv \prod_{i=1}^n[v_1,v_{\alpha(x_i)}]^i\pmod{Q}.\]
This is a relation in the linearly independent generators $[v_1,v_g]$ of the elementary abelian group $P'$.
Notice that $\alpha(x_i)\ne 1$ and $n<|G|<p$.
Comparing exponents reveals $\alpha(x_i)=x_i$ for $i=1,\ldots,n$. Since $G=\langle x_1,\ldots,x_n\rangle$, $\alpha$ induces the identity automorphism on $G$. 
\end{proof}

How does the group $P$ in the proof above relate to Heineken--Liebeck's construction?
Recall that Frucht's theorem states that every finite group $G$ is the automorphism group of some finite graph $\mathcal{G}$. Frucht's graph is based on the Cayley color graph $\mathcal{C}$, which depends on a generating set $x_1,\ldots,x_n$ of $G$. More precisely, the vertex set of $\mathcal{C}$ is $\{v_g:g\in G\}$ and there is an arrow $v_g\to v_h$ of color $i$ if and only if $h=gx_i$. Our proof of \autoref{thmext} rests on the elementary fact that the color-preserving automorphism group of $\mathcal{C}$ is isomorphic to $G$. 
The purpose of the group $Q$ is to enforce automorphisms to permute the generators of $P$.

\begin{proof}[Proof of Pettet's theorem]
The implications (1)$\Rightarrow$(2) and (1)$\Rightarrow$(3) are obvious. Conversely, if $\alpha$ fulfills (2) or (3), then $\alpha$ extends/lifts to the group $\hat{G}$ constructed in \autoref{thmext}. The theorem implies that $\alpha$ is inner. 
\end{proof}

If $G$ is solvable, the proof above shows that (2) and (3) of Pettet's theorem can be restricted to solvable extension groups $\hat{G}$.

\section{Realizing outer automorphisms}\label{secOut}

We start by proving \cite[Lemma~2]{Cornulier}. Let $n\in\NN$ and $p>n$ be a prime. 
Let $C_p\rtimes C_{p-1}$ be the holomorph of $C_p$, i.\,e. $C_{p-1}$ acts faithfully on $C_p$. Let $S_n$ be the symmetric group of degree $n$.

\begin{Lem}\label{lemOutHol}
We have $\Out((C_p\rtimes C_{p-1})^n)\cong S_n$.
\end{Lem}
\begin{proof}
The argument is similar as for the group $(C_q\rtimes C_p)^n$ in the proof of \autoref{thmext}, although the statement requires the factor $p-1$. 
Let $P=\langle x_1,\ldots,x_n\rangle\cong C_p^n$ and $Q:=\langle y_1,\ldots,y_n\rangle\cong C_{p-1}^n$ such that each $y_i$ induces an automorphism of $\langle x_i\rangle$ of order $p-1$ and centralizes $x_j$ for $j\ne i$. 
It is obvious that every permutation $\pi\in S_n$ induces an automorphism $\alpha_\pi$ of $G:=P\rtimes Q$ by permuting the factors $\langle x_i,y_i\rangle$. If $\pi\ne\id$, then $\alpha_\pi$ is not inner, because it acts non-trivially on the abelian quotient $G/P\cong Q$. Hence, $S_n$ induces a subgroup of $\Out(G)$. 

Conversely, let $\alpha\in\Aut(G)$ be arbitrary. Then $\alpha$ normalizes the normal Sylow $p$-subgroup $P$ of $G$.
By the Schur--Zassenhaus theorem, we can further assume that $\alpha(Q)=Q$. 
Since 
\[|\C_G(\alpha(x_i))|=|\C_G(x_i)|=p^n(p-1)^{n-1},\] 
there exists some $\pi\in S_n$ with $\alpha(x_i)\in\langle x_{\pi(i)}\rangle$ for $i=1,\ldots,n$. 
By composing $\alpha$ with $\alpha_{\pi}^{-1}$, we may assume that $\pi=\id$. A similar argument yields $\alpha(y_i)\in\langle y_i\rangle$. By \autoref{lemaut} applied to $\langle x_i\rangle\rtimes\langle y_i\rangle$, we have $\alpha(y_i)=y_i$. 
Since $\Aut(\langle x_i\rangle)\cong\langle y_i\rangle$, the action of $\alpha$ on $\langle x_i\rangle$ is induced by conjugation with some power of $y_i$. 
Composing $\alpha$ with the corresponding inner automorphism, gives $\alpha(x_i)=x_i$ (this will not affect the action of $\alpha$ on $\langle x_j,y_j\rangle$ for $j\ne i$). Doing this for $i=1,\ldots,n$, leads to $\alpha=\id$.
\end{proof}

Let $F$ be the free group in the free generators $x_1,\ldots,x_n$. Let $F^p=\langle x^p:x\in F\rangle^F\unlhd F$ be the normal closure of the set of all $p$-powers in $F$. Then $\overline{F}:=F/F^p$ is the free group of rank $n$ and exponent $p$.
We will identify $x_i$ with its image in $\overline{F}$. Define the lower central series by $\overline{F}^{[1]}:=\overline{F}$ and 
\[\overline{F}^{[k+1]}:=[\overline{F},\overline{F}^{[k]}]=\langle [x,y]:x\in \overline{F},\ y\in \overline{F}^{[k]}\rangle\]
for $k\ge 1$ as usual. We call $\overline{F}_c:=\overline{F}/\overline{F}^{[c+1]}$ the free group of rank $n$, exponent $p$ and nilpotency class $c$ (we will see in \autoref{lemLie} that the class of $\overline{F}_c$ cannot be smaller than $c$).
The adjective “free” is justified by the following universal property: If $G$ is a nilpotent group of exponent $p$ and class $\le c$, and if $y_1,\ldots,y_n\in G$, then there exists a (unique) homomorphism $\overline{F}_c\to G$, $x_i\mapsto y_i$ for $i=1,\ldots,n$. 
We will use this principle frequently in order to verify certain commutator relations in $\overline{F}_c$. 
 
It is easy to show that $\overline{F}^{[k]}/\overline{F}^{[k+1]}$ is generated by the cosets of the $k$-fold commutators $[x_{i_1},\ldots,x_{i_k}]$ where $1\le i_1,\ldots,i_k\le n$ (see \cite[Hilfssatz~III.1.11]{Huppert}). It follows that $\overline{F}^{[k]}/\overline{F}^{[k+1]}$ is a finite elementary abelian $p$-group.
In particular, $\overline{F}_c$ is a finite group. The following lemma is certainly known, but I was unable to find a proper reference.

\begin{Lem}\label{lemMultlinear}
For every finite group $G$ of exponent $p$ and $k\in\NN$, the map 
\begin{align*}
\Phi_k:(G/G')^k&\to G/G^{[k+1]},\\
(g_1G',\ldots,g_kG')&\mapsto[g_1,\ldots,g_k]G^{[k+1]}
\end{align*}
is well-defined and multilinear over $\FF_p$, i.e. for $1\le i\le k$, $h\in G$ and $h'\in G'$, 
\[[g_1,\ldots,g_{i-1},g_ihh',g_{i+1},\ldots,g_k]\equiv[g_1,\ldots,g_k][g_1,\ldots,h,\ldots,g_k]\pmod{G^{[k+1]}}.\]
\end{Lem}
\begin{proof}
If we interpret $[x_1]$ as $x_1$, then $\Phi_1$ becomes the identity map. Now let $k\ge 2$. Recall that $[G^{[s]},G^{[t]}]\le G^{[s+t]}$ by \cite[Hauptsatz~III.2.11]{Huppert}. 
A direct computation shows $[xy,z]=[x,y,z][y,z][x,z]$ and $[z,xy]=[xy,z]^{-1}=[z,x][z,y][x,y,z]^{-1}$ for every $x,y,z\in G$. 
If $i=1$, we put $z:=[g_2,\ldots,g_k]\in G^{[k-1]}$ and obtain
\begin{align*}
[g_1hh',g_2,\ldots,g_k]&=[g_1h,h',z][h',z][g_1h,z]\equiv [g_1h,z]\equiv [g_1,h,z][h,z][g_1,z]\\
&\equiv[g_1,z][h,z]\equiv[g_1,\ldots,g_k][h,g_2,\ldots,g_k]\pmod{G^{[k+1]}}.
\end{align*}
It remains to consider the component $i\ge 2$. Let $z:=[g_{i+1},\ldots,g_k]$. By induction on $k$, we derive
\[[g_2,\ldots,g_ihh',,\ldots,g_k]\equiv [g_2,\ldots,g_i,z][g_2,\ldots,h,z]\pmod{G^{[k]}}\]
and
\[[g_1,\ldots,g_ihh',\ldots,g_k]\equiv [g_1,\ldots,g_i,z][g_1,\ldots,h,z]\pmod{G^{[k+1]}}.\qedhere\]
\end{proof}

The following result resembles \cite[Lemma~5]{Cornulier}.\footnote{Since the proof of \cite[Lemma~5]{Cornulier} takes place in the non-nilpotent Lie algebra $\mathfrak{gl}_n$, it is not clear to me that the obtained result can actually be used to prove the main theorem.}

\begin{Lem}\label{lemLie}
Let $n\ge 3$, $\pi\in S_{n-1}$ and $0\le a\le p-1$ such that
\begin{equation}\label{eqComm1}
[x_1,\ldots,x_{n-1},x_1]\equiv [x_{\pi(1)},\ldots,x_{\pi(n-1)},x_{\pi(1)}]^a\pmod{\overline{F}^{[n+1]}}.
\end{equation}
Then $\pi=\id$ and $a=1$. 
\end{Lem}
\begin{proof}
By the universal property, it suffices to prove the claim for any elements $x_1,\ldots,x_{n-1}$ of a group $P$ with exponent $p$ and nilpotency class $\le n$. Let $P\le\GL(n+1,p)$ be the group of upper unitriangular matrices. For $x\in P$ we have 
\[x^p-1=(x-1)^p=(x-1)^{n+1}(x-1)^{p-n-1}=0\] 
since $p>n$. Hence, $P$ has exponent $p$. For $i<j$, let $E_{ij}\in P$ be the unitriangular matrix with $1$ on position $(i,j)$ and zero elsewhere off the diagonal. A direct calculation shows that 
\begin{equation}\label{eqEij}
[E_{ij},E_{kl}]=E_{il}^{\delta_{jk}}E_{kj}^{-\delta_{il}},
\end{equation}
where $\delta_{jk}\delta_{il}=0$ since $i<j$ and $k<l$. 
An induction shows that $P^{[k]}$ is generated by the matrices $E_{ij}$ with $|j-i|\ge k$. In particular, $P^{[n]}=\langle E_{1,n+1}\rangle\cong C_p$ and $P^{[n+1]}=1$ (see \cite[Satz~III.16.3]{Huppert}). So $P$ has indeed nilpotency class $n$ . 
We define $x_1:=E_{12}E_{n,n+1}$ and $x_i:=E_{i,i+1}$ for $i=2,\ldots,n-1$. Then the right hand side of \eqref{eqComm1} is
\[[x_1,\ldots,x_{n-1},x_1]=\begin{cases}
[x_1,\ldots,x_{n-2},E_{n-1,n+1}]=\ldots=[x_1,E_{2,n+1}]=E_{1,n+1}&\text{if }n\ge 4,\\
[x_1,E_{24}E_{13}^{-1}]=[E_{12}E_{34},E_{24}][E_{12}E_{34},E_{13}]^{-1}=E_{14}&\text{if }n=3.
\end{cases}
\]
Suppose first that $\pi(1)\ne1$. Then $x_1$ appears only once in $c:=[x_{\pi(1)},\ldots,x_{\pi(n-1)},x_{\pi(1)}]$. By \autoref{lemMultlinear}, $c\equiv c_1c_2\pmod{\overline{F}^{[n+1]}}$, where $c_1$ and $c_2$ are $n$-fold commutators in the elements $E_{i,i+1}$. Both $c_i$ contain $x_{\pi(1)}$ twice, so they muss miss some $E_{r,r+1}$. But now each $c_i$ lives inside a direct product of the form 
\[Q:=\biggl\{\begin{pmatrix}
Q_1&0\\0&Q_2
\end{pmatrix}:Q_1\le\GL(r,p),\ Q_2\le\GL(n+1-r,p)\biggr\}.\] 
Since $Q$ has nilpotency class $<n$, we derive the contradiction $c\equiv c_1c_2\equiv 1\pmod{\overline{F}^{[n+1]}}$. 

Therefore, $\pi(1)=1$. Now $[x_{\pi(n-1)},x_1]\ne 1$ implies $\pi(n-1)\in\{2,n-1\}$ by \eqref{eqEij}. Assume first that $\pi(n-1)=n-1$. Then $[x_{\pi(n-2)},x_{n-1},x_1]=[x_{\pi(n-2)},E_{n-1,n+1}]\ne 1$ implies $\pi(n-2)=n-2$. Inductively, one obtains $\pi=\id$, $c=E_{1,n+1}$ and $a=1$ in this case. Now suppose that $\pi(n-1)=2$ and without loss of generality, $n\ge 4$. 
Here we use a different realization of $\overline{F}_n$ inside $P$.
More precisely, we reassign $x_i:=E_{12}E_{23}\ldots E_{n-1,n}$ for $i=1,\ldots,n-2$ and $x_{n-1}:=E_{n,n+1}$. Then clearly, $c=[\ldots,[x_1,x_1]]=1$. On the other hand, the right hand side of \eqref{eqComm1} becomes
\[[x_1,\ldots,x_{n-1},x_1]=[x_1,\ldots,x_1,E_{n-1,n+1}]^{-1}=\ldots=[x_1,E_{2,n+1}]^{-1}=E_{1,n+1}^{-1}.\]
Contradiction. 
\end{proof}

We have duplicated $x_1$ in the commutator in \autoref{lemLie} to avoid relations of the form
\[[*,\ldots,*,x,y]\equiv[*,\ldots,*,y,x]^{-1}\pmod{\overline{F}^{[n+1]}}.\]

To prove Cornulier's theorem, let $G=\{g_1,\ldots,g_n\}$ be a finite group of order $n$. 
We construct a finite group $H$ with $\Out(H)\cong G$. Since $\Out(1)=1$ and $\Out(C_3)\cong C_2$, we may assume that $n\ge 3$ (as in \autoref{lemLie}).
We identify the generators $x_i$ of $\overline{F}$ with $x_{g_i}$ and define
\[N:=\langle[x_{hg_1},\ldots,x_{hg_{n-1}},x_{hg_1}]:h\in G\rangle \overline{F}^{[n+1]}\le \overline{F}^{[n]}.\]
Since $\overline{F}^{[n]}/\overline{F}^{[n+1]}\le \Z(\overline{F}_n)$, it follows that $N\unlhd \overline{F}$. Let $P:=\overline{F}/N\cong \overline{F}_n/(N/\overline{F}^{[n+1]})$. Notice that $P$ has exponent $p$ and nilpotency class $\le n$. Moreover, $P/P'\cong\overline{F}/\overline{F}'\cong C_p^n$. Again we will identify the $x_i$ with their images in $P$.

Let $\FF_p^\times=\langle\zeta\rangle$. For $1\le i\le n$, the map $x_j\mapsto x_j^{\zeta^{\delta_{ij}}}$ can be extended to an automorphism $q_i$ of $\overline{F}$. By \autoref{lemMultlinear},
\[q_i([x_{j_1},\ldots,x_{j_n}])=[q_i(x_{j_1}),\ldots,q_i(x_{j_n})]\equiv[x_{j_1},\ldots,x_{j_n}]^\gamma\pmod{\overline{F}^{[n+1]}}\]
for some $\gamma\in\ZZ$. In particular, $q_i(N)=N$ and $q_i$ extends to an automorphism of $P$. Moreover, the group $Q:=\langle q_1,\ldots,q_n\rangle\le\Aut(P)$ is isomorphic to $C_{p-1}^n$. Finally, we define $H:=P\rtimes Q$. As usual, we regard $P$ and $Q$ as subgroups of $H$. Then $q_ix_jq_i^{-1}=x_j^{\zeta^{\delta_{ij}}}$ for $1\le i,j\le n$. Note that $H/P'\cong (C_p\rtimes C_{p-1})^n$. 

The following result implies Cornulier's theorem.

\begin{Thm}
With the notation above, $\Out(H)\cong G$. 
\end{Thm}
\begin{proof}
For $h\in G$, the map $x_i\mapsto x_{hg_i}$ ($i=1,\ldots,n$) can be extended to an automorphism $\alpha_h$ of $\overline{F}$. By the definition of $N$, we have $\alpha_h(N)=N$. Therefore, we consider $\alpha_h$ as an automorphism of $P$. 
There is a similar automorphism $\beta_h\in\Aut(Q)$ with $\beta_h(q_i)=q_{hg_i}$, where $q_i$ is identified with $q_{g_i}$. 
Since 
\[\alpha_h(q_ix_jq_i^{-1})=\alpha_h(x_j)^{\zeta^{\delta{ij}}}=x_{hg_j}^{\zeta^{\delta{ij}}}=q_{hg_j}x_{hg_i}q_{hg_j}^{-1}=\beta_h(q_j)\alpha_h(x_i)\beta_h(q_j)^{-1},\]
the actions are compatible. This gives rise to a regular action $\alpha:G\to\Aut(H)$. 
Since $g\ne 1$ acts non-trivially on $H/P\cong Q$, $\alpha(G)\cap\Inn(H)=1$. Thus, it suffices to show that $\Aut(H)=\alpha(G)\Inn(H)$.

To this end, let $\gamma\in\Aut(H)$ be arbitrary. Then $\gamma$ normalizes the normal Sylow $p$-subgroup $P$ and $P'$. 
By \autoref{lemOutHol}, we may assume that $\gamma$ permutes the factors of $H/P'$. So there exists a permutation $\pi\in S_n$ such that $\gamma(q_i)=q_{\pi(i)}\pmod{P'}$ and $\gamma(x_i)\equiv x_{\pi(i)}\pmod{P'}$ for $i=1,\ldots,n$. This implies 
\[[x_{\pi(1)},\ldots,x_{\pi(n-1)},x_{\pi(1)}]=\gamma([x_1,\ldots,x_{n-1},x_1])=\gamma(1)=1\]
by \autoref{lemMultlinear}. This yields an equation inside $N/\overline{F}^{[n+1]}$:
\[[x_{\pi(1)},\ldots,x_{\pi(n-1)},x_{\pi(1)}]\equiv \prod_{h\in G}[x_{hg_1},\ldots,x_{hg_{n-1}},x_{hg_1}]^{a_h}\pmod{\overline{F}^{[n+1]}}\]
for some $0\le a_h\le p-1$. By the universal property, this equation remains true when we set $x_{\pi(n)}=1$. For the unique $h\in G$ with $x_{hg_n}=x_{\pi(n)}$ we deduce 
\[[x_{\pi(1)},\ldots,x_{\pi(n-1)},x_{\pi(1)}]\equiv[x_{hg_1},\ldots,x_{hg_{n-1}},x_{hg_1}]^{a_h}\pmod{\overline{F}^{[n+1]}}.\]
By \autoref{lemLie}, $x_{\pi(i)}=x_{hg_i}$ for $i=1,\ldots,n-1$. Therefore, after composing $\gamma$ with $\alpha(h)^{-1}$, we may assume that $\pi=1$. Since $|P'|$ is coprime to $|Q|$ and $\gamma(Q)\le P'Q$, there exists a $y\in P'$ with $\gamma(Q)=yQy^{-1}$ by the Schur--Zassenhaus theorem. Since conjugation with $y$ does not affect $H/P'$, we may assume that $\gamma(Q)=Q$.
In particular, $\gamma$ centralizes $Q$. 

Each quotient $P^{[k]}/P^{[k+1]}$ has a basis (as an elementary abelian group) consisting of some $k$-fold commutators in the $x_i$. By concatenation we obtain a basis $c_1,\ldots,c_s$ of $\bigtimes_{k=1}^n P^{[k]}/P^{[k+1]}$. 
For $c_i\in P^{[k]}\setminus P^{[k+1]}$, we have $q_jc_iq_j^{-1}\equiv c_i^{\zeta^l}\pmod{P^{[k+1]}}$ by \autoref{lemMultlinear}, where $l$ is the multiplicity of $x_j$ as a component of $c_i$. Clearly, $l\le k-1\le n-1<p-1$. Hence, $\zeta^l\equiv 1\pmod{p}$ can only hold if $l=0$. This shows that $q_j$ centralizes $c_i$ if and only if $x_j$ does not appear in $c_i$. Since the $c_i$ form a basis, it follows that $\C_P(q_j)=\langle x_i:i\ne j\rangle$.
On the other hand, $q_j\gamma(x_i)q_j^{-1}=\gamma(q_jx_iq_j^{-1})=\gamma(x_i)$ for $j\ne i$ shows that 
\[\gamma(x_i)\in \bigcap_{j\ne i}\C_P(q_j)=\langle x_i\rangle.\] 
Since we already know that $\gamma(x_i)\equiv x_i\pmod{P'}$, we conclude $\gamma(x_i)=x_i$ for $i=1,\ldots,n$ and $\gamma=\id$, as desired.
\end{proof}

In order to estimate $|H|$ in terms of $n$, we consider the free nilpotent Lie algebra $L$ over $\QQ$ of rank $n$ and class $n$.
By Witt's formula, 
\[\dim L=\sum_{k=1}^n\sum_{d\mid k}\mu(d)n^{k/d},\]
where $\mu$ is the Möbius function. 
This number grows roughly as $n^{n-1}$ (see \cite[Lemma~20.7]{BNV}).
Notice that $\FF_p\otimes L$ is the corresponding free nilpotent Lie algebra over $\FF_p$. Since $p>n$, the Lazard correspondence turns $\FF_p\otimes L$ into $\overline{F}_n$ (see \cite[Example~10.24]{Khukhro}). In particular, $|\overline{F}_n|=p^{\dim L}$. Moreover, an application of \autoref{lemLie} reveals that $|N/\overline{F}^{[n+1]}|=p^n\sim (p-1)^n=|Q|$. Altogether, the order of magnitude of $|H|$ is $p^{n^{n-1}}$ (the estimate $n^n$ in \cite{Cornulier} is unjustified). 

As a concrete example, consider $G\cong C_3$. Here we can take $p=5$. Then $|\overline{F}_3|=5^{14}$, $|P|=5^{11}$ and $|H|=2^65^{11}$. Cornulier's construction yields a group of order $2^65^{29}$, as he remarked at the end of \cite{Cornulier}. 

One may ask if the group $H$ can also be used to prove Pettet's theorem. This does not seem to be easy, since it is not clear whether every automorphism of $H\rtimes G$ normalizes $H$. Conversely, the group $N$ in \autoref{thmext} has outer automorphisms, which do not come from $G$.

\end{document}